\documentclass[psamsfonts]{amsart}
\usepackage{amssymb,amsfonts,lacromay}
\usepackage[all,arc]{xy}
\usepackage{enumerate}
\usepackage{lscape,pstricks,pst-plot,color}

\usepackage{mathrsfs}
\newtheorem{thm}{Theorem}[section]
\newtheorem{cor}[thm]{Corollary}

\newtheorem{lem}[thm]{Lemma}

\theoremstyle{definition}

\newtheorem{rem}[thm]{Remark}

\newcommand{\Z}{\mathbb{Z}}

\makeatletter
\let\c@equation\c@thm
\makeatother
\numberwithin{equation}{section}

\makeatletter
\ifx\SK@label\undefined\let\SK@label\label\fi
 \let\your@thm\@thm
 \def\@thm#1#2#3{\gdef\currthmtype{#3}\your@thm{#1}{#2}{#3}}
 \def\mylabel#1{{\let\your@currentlabel\@currentlabel\def\@currentlabel
  {\currthmtype~\your@currentlabel}
 \SK@label{#1@}}\label{#1}}
 
\makeatother

\newgray{gridline}{0.8}
\newrgbcolor{tauzerocolor}{0.5 0.5 0.5}

\newrgbcolor{hzerotowercolor}{0.5 0.5 0.5}
\newrgbcolor{tauonecolor}{1 0 0}
\newrgbcolor{hiddenhzerocolor}{0 0 1}
\newrgbcolor{hiddenhonecolor}{0 0 1}
\newrgbcolor{hiddenhtwocolor}{0 0 1}

\newrgbcolor{dtwocolor}{0 0.7 0.7}
\newrgbcolor{dthreecolor}{1 0.5 0}

\title{Motivic stable homotopy and the stable 51 and 52 stems}

\author{Daniel C. Isaksen}
\address{Department of Mathematics, Wayne State University, Detroit, MI 48202}
\email{isaksen@wayne.edu}
\author{Zhouli Xu}
\address{Department of Mathematics, The University of Chicago, Chicago, IL 60637}
\email{xu@math.uchicago.edu}

\thanks{The second author was supported by NSF grant DMS-1202213.}



\begin{document}

\begin{abstract}
We establish a differential $d_2(D_1)=h_0^2h_3g_2$ in the $51$-stem of the Adams spectral sequence at the prime $2$,
which gives the first correct calculation of the stable
51 and 52 stems.
This differential is remarkable since we know of no
way to prove it without recourse to the motivic
Adams spectral sequence.  It is the last undetermined
differential in the range of the first author's detailed
calculations of the $n$-stems for $n<60$ \cite{DI1}.
%, and it is the most refractory one. 
This note advertises the use
of the motivic Adams spectral sequence to obtain
information about classical stable homotopy groups.
\end{abstract}

\maketitle

\section{Introduction}

The problem of computing the stable homotopy groups of spheres is of fundamental importance in algebraic topology. Although this subject has been studied for a very long time, the Adams spectral sequence is still the best way to do stemwise computations at the prime 2.

Bruner produced a computer-generated table of $d_2$ differentials in the Adams
spectral sequence \cite{Br3}.
He started with his theorem on the interaction between Adams differentials
and squaring operations \cite{BMMS} to obtain several values of the $d_2$ differential.
Then he methodically exploited all (primary) multiplicative relations, both
forwards and backwards, to deduce more and more values of the $d_2$ differential.
Through at least the 80-stem, this procedure gives the vast majority of the
values of the $d_2$ differential.

In Bruner's approach, the first unknown value is $d_2(D_1)$,
where $D_1$ is a certain element in the 52-stem.
This gives a precise sense to the claim that $d_2(D_1)$ is
``harder" than any previous $d_2$ differential.

Mark Mahowald communicated (unpublished) the following argument for the presence of the differential $d_2(D_1)=h_0^2h_3g_2$.  First, there is an element $x$ in the $37$-stem that detects $\sigma \theta_4$
\cite[Proposition 3.5.1]{BMT}.  Since $h_3^2 x = h_0^2 h_3 g_2$,
we deduce that $h_0^2 h_3 g_2$ detects $\sigma^3 \theta_4$.
Now $\sigma^3 = \nu \nu_4$, so
$h_0^2 h_3 g_2$ detects $\nu \nu_4 \theta_4$.
The final step is the claim that $\nu_4 \theta_4$ is zero.

This final step turns out to be false;
$\nu_4 \theta_4$ is detected by $h_2 h_5 d_0$ \cite[Lemma 4.2.90]{DI1}.
Mahowald's mistake arose from an incorrect understanding of the Toda
bracket $\langle \theta_4, 2, \sigma^2  \rangle$ \cite[Lemma 4.2.91]{DI1}.

In this note, we will repair the hole left by this mistake.
The proof relies on motivic calculations in a fundamental way.
In other words, our argument does not work if applied in the classical context.
We will first establish a Massey product involving the element $D_1$.
Then we will apply the higher Leibniz rule on the interaction of
Adams differentials with Massey products \cite{Moss}.
The second author discovered this proof
based on motivic calculations by the first author.
We adopt notation from \cite{DI1} without further explanation.
For charts of the classical and
motivic Adams spectral sequences, see \cite{DI2}.

The higher Leibniz rule of Moss \cite{Moss} is an appealing result
that has not been as useful in practice as one might expect.
In applications, indeterminacies often interfere.
The differential that we establish here is an unusually clear application
of the higher Leibniz rule.
Our work suggests that an extension of Bruner's program to methodically
exploit all $3$-fold Massey products is likely to lead to further
new results about Adams $d_2$ differentials.

The origin of this note lies in the first author's analysis of the Adams
spectral sequence through the 59-stem.  The first author was able to
give careful arguments for every Adams differential, with the sole exception
of possible differentials on the element $D_1$.
The second author was instrumental in finishing the last remaining differential,
whose argument we present here.

The chief consequence of our Adams differential calculation is the
following results about the stable 51 and 52 stems.

\begin{thm}
The 2-primary order of the stable 51-stem is 128.
As a group, the 2-primary stable 51-stem is either
$\Z/8 \oplus \Z/8 \oplus \Z/2$ or
$\Z/8 \oplus \Z/4 \oplus \Z/2 \oplus \Z/2$.
\end{thm}

\begin{proof}
This follows immediately from the Adams $E_\infty$-page.
The first $\Z/8$ lies in the image of $J$ and has a generator
that is detected by $P^6 h_2$.
The last $\Z/2$ is detected by $h_2 B_2$.

The remaining elements are detected by $h_3 g_2$, $h_0 h_3 g_2$,
and $g n$.  These elements assemble into $\Z/8$ if
there is a hidden $2$ extension from $h_0 h_3 g_2$ to $g n$,
and they assemble into $\Z/4 \oplus \Z/2$ if there is no
hidden $2$ extension.
\end{proof}

\begin{rem}
According to \cite{Ko}, there is no hidden $2$ extension in the 
stable 51-stem.  However, this claim is inconsistent with other
claims in \cite{Ko}, as discussed in \cite[Remark 4.1.17]{DI1}.
Because of this uncertainty, we leave this 2 extension unresolved.
\end{rem}

\begin{thm}
The 2-primary stable 52-stem is equal to $\Z/2 \oplus \Z/2 \oplus \Z/2$.
\end{thm}

\begin{proof}
This is immediate from the Adams $E_\infty$-page.
According to \cite{DI1}, there are no possible hidden $2$ extensions.
\end{proof}

%\subsection{Acknowledgement}
%This note is a sign of our gratitude to Mark Mahowald for his tenacious exploration of the stable stems and his generosity towards us. The authors would like to dedicate this note to him, with special thanks for his inspiring weekly careful instruction and his guidance to the second author during
%the year before his untimely death.
%In particular, Mahowald gave the differential that we consider here
%as an exercise to the second author.

\section{The differential}

\vskip \baselineskip

The crucial calculation is the following lemma which is due to Tangora \cite{Tan}.

\begin{lem}
\label{lem:G-bracket}
In the classical Adams $E_2$ page, we have a Massey product
$$G=\langle h_1,h_0,D_1\rangle .$$
\end{lem}

\begin{proof}
This Massey product is proven using the Lambda algebra. Since the elements $D_1$
and $G$ lie beyond the range of the published Curtis tables \cite{CGMM}, Tangora used another Massey product to
deduce that the representative of $D_1$ in the Lambda algebra has a leading term
$\lambda_{4}\lambda_{7}\lambda_{11}\lambda_{15}\lambda_{15}$. One can also use
unpublished results of Mahowald and Tangora to see directly that $G$ is represented by
$\lambda_{2}\lambda_{4}\lambda_{7}\lambda_{11}\lambda_{15}\lambda_{15}$. Then the Massey product
follows from $d(\lambda_{2})=\lambda_{1}\lambda_{0}$.
\end{proof}

\begin{rem} One cannot apply May's convergence theorem \cite{May}
directly to get this Massey product from
the May differential $d_2(h_2b_{22}b_{40})=h_0D_1$.
The problem is that there is a ``crossing" May differential
$d_4(h_1 b_{3,1}^2 ) = h_1 h_3 g_2 + h_1 h_5 g$
that voids the hypotheses of May's convergence theorem.
\end{rem}

\begin{rem}
Bruner has verified this Massey product using computer data.
\end{rem}

The classical calculation of Lemma \ref{lem:G-bracket}
allow us to deduce the analogous motivic calculation.

\begin{cor}
\label{cor:G-bracket}
In the motivic Adams $E_2$ page, we have a Massey product
$$\tau G=\langle h_1,h_0,D_1\rangle .$$
\end{cor}

\begin{proof}
This is immediate from the fact that classical calculations can be recovered
from motivic calculations by inverting $\tau$ \cite{DI1}.
\end{proof}

%We will need the following Adams $d_2$ differential, which is
%proved in \cite[Lemma 3.3.12]{DI1}.
%
%\begin{lem}
%In the motivic Adams spectral sequence, $d_2(\tau G)=h_5c_0d_0$.
%\end{lem}

\begin{thm}
\label{thm:main}
In the motivic Adams spectral sequence, $d_2(D_1)=h_0^2h_3g_2$.
\end{thm}

\begin{proof}
Corollary \ref{cor:G-bracket} gives us
the Massey product $\tau G=\langle h_1,h_0,D_1\rangle $ in the motivic Adams spectral sequence.  Using the higher Leibniz rule of Moss \cite{Moss}, we have $d_2(\tau G)=\langle h_1,h_0,d_2(D_1)\rangle $ since there is no indeterminacy.
Since $d_2(\tau G)=h_5c_0d_0$ \cite[Lemma 3.3.12]{DI1}, we have $\langle h_1,h_0,d_2(D_1)\rangle =h_5c_0d_0$. In particular, this is nonzero. Then
the only possibility is that $d_2(D_1)=h_0^2h_3g_2$.
\end{proof}

\begin{rem}
Note that $h_0^2h_3g_2=h_2^2h_5d_0$,
so
\[
\langle h_1, h_0, h_0^2 h_3 g_2 \rangle =
\langle h_1, h_0, h_2^2 h_5 d_0 \rangle =
\langle h_1, h_0, h_2^2 \rangle h_5 d_0 = h_5 c_0 d_0,
\]
as dictated by the higher Leibniz rule.
\end{rem}

\begin{cor}
In the classical Adams spectral sequence, $d_2(D_1)=h_0^2h_3g_2$.
\end{cor}

\begin{proof}
This is immediate from the fact that classical calculations can be recovered
from motivic calculations by inverting $\tau$ \cite{DI1}.
\end{proof}

\begin{rem}
In the classical Adams spectral sequence, $h_5c_0d_0$ and $d_2(G)$ are both zero.
This means that our proof is strictly motivic in nature.
We do not know a proof using only the classical Adams spectral sequence.
\end{rem}

\begin{rem}
Using only classical information,
we suggest a possible argument that $D_1$ does not survive the classical
Adams spectral sequence.
Start with the Massey product $G=\langle h_1,h_0,D_1\rangle $
and the classical differential on $d_3(G) = P h_5 d_0$.
Moss's convergence theorem \cite{Moss} then implies that
$\langle \eta, 2, \{D_1\} \rangle $ is not a well-formed Toda bracket. The only possibility is that $D_1$ does not survive. However, that still leaves two possible differentials: $d_2(D_1) = h_0^2 h_3 g_2$, or $d_4(D_1) = g n$. We do not know how to rule out
the second alternative using only the classical Adams spectral sequence.
\end{rem}

\begin{rem}
Mahowald's original argument for
the differential $d_2(D_1)$ used a faulty computation of the
Toda bracket $\langle \theta_4, 2, \sigma^2\rangle$.  Partial information
about this Toda
bracket is used in the construction of a Kervaire invariant element $\theta_5$ in dimension $62$ \cite{BJM}.
Our improved understanding of the Toda bracket does not contradict
any of the claims of \cite{BJM}.
%Our determination of $d_2(D_1)$ does not depend on the computation of the Toda bracket $\langle \theta_4, 2, \sigma^2\rangle $.
\end{rem}


\begin{thebibliography}{99}

\bibitem{BJM}
M. G. Barratt, J. D. S. Jones and M. E. Mahowald.
Relations amongst Toda brackets and the Kervaire invariant in dimension $62$.
J. London Math. Soc. 30 (1984), 533--550.

\bibitem{BMT}
M. G. Barratt, M. E. Mahowald and M .C. Tangora.
Some differentials in the Adams spectral sequence. II
Topology. 9 (1970), 309--316.

\bibitem{BMMS}
R. R. Bruner, J. P. May, J. E. McClure, and M. Steinberger.
$H\sb \infty $ ring spectra and their applications.
Lecture Notes in Math., 1176, Springer, Berlin, 1986.

\bibitem{Br3}
R. R. Bruner.
The Adams $d_2$ differential.
Preprint (2010).

\bibitem{CGMM}
E. B. Curtis, P. Goerss, M. Mahowald and R. J. Milgram.
Calculations of unstable Adams $E_2$ terms for spheres.
Algebraic topology (Seattle, Wash., 1985), 208--266, Lecture Notes in Math., 1286,
Springer, Berlin, 1987.

\bibitem{DI1}
D. C. Isaksen.
Stable stems.
arXiv:1407.8418.

\bibitem{DI2}
D. C. Isaksen.
Classical and motivic Adams charts.
arXiv:1401.4983.



\bibitem{Ko}
S. O. Kochman.
Stable homotopy groups of spheres, a computer-assisted approach.
Lecture Notes in Mathematics 1423, Springer-Verlag, 1990.
%
%\bibitem{Ko2}
%Stanley O. Kochman.
%Uniqueness of Massey products on the stable homotopy of spheres.
%Canandian Journal of Mathematics, Vol. XXXII, No. 3, pp. 576-589. June, 1980.

%\bibitem{Lawrence}
%A.\ Lawrence,
%Higher order compositions in the Adams spectral sequence.
%Bull.\ Amer.\ Math.\ Soc.\ 76 (1970), 874-?877. 

\bibitem{May}
J. P. May.
Matric Massey products.
J. Algebra 12 (1969), 533--568.

%\bibitem{Mi}
%Mamoru Mimura.
%On the generalized Hopf homomorphism and higher composition, Part I.
%J. Math. Kyoto. Univ. 4-1 (1964), 171-190.

\bibitem{Moss}
R. M. F. Moss.
Secondary compositions and the Adams spectral sequence.
Math. Z. 115 (1970), 283--310.

\bibitem{Tan}
M. Tangora.
Some Massey products in Ext.
Topology and representation theory (Evanston, IL, 1992), 269--280,
Contemp. Math., 158, Amer. Math. Soc., Providence, RI, 1994.

%\bibitem{Toda}
%Toda's book.

\end{thebibliography}
\end{document}